\documentclass[twoside]{article}

\usepackage{blindtext} 

\usepackage[sc]{mathpazo} 
\usepackage[T1]{fontenc} 
\usepackage{microtype} 

\usepackage[english]{babel} 

\usepackage[hmarginratio=1:1,top=32mm,columnsep=20pt]{geometry} 
\usepackage[hang, small,labelfont=bf,up,textfont=it,up]{caption} 
\usepackage{booktabs} 

\usepackage{lettrine} 

\usepackage{enumitem} 
\setlist[itemize]{noitemsep} 

\usepackage{abstract} 

\usepackage{titlesec} 
\renewcommand\thesection{\Roman{section}} 
\renewcommand\thesubsection{\roman{subsection}} 
\titleformat{\section}[block]{\large\scshape\centering}{\thesection.}{1em}{} 
\titleformat{\subsection}[block]{\large}{\thesubsection.}{1em}{} 

\usepackage{fancyhdr} 
\usepackage{titling} 

\usepackage{hyperref} 

\pretitle{\begin{center}\Huge\bfseries} 
\posttitle{\end{center}} 
\title{Controllability of the semilinear wave equation governed by a multiplicative
control} 
\author{%
\textsc{M. Ouzahra}\\
\normalsize MASI Team, University of
Sidi Mohamed Ben Abdellah \\ 
\normalsize \href{mohamed.ouzahra@usmba.ac.ma}{mohamed.ouzahra@usmba.ac.ma} 
}
\date{}
\renewcommand{\maketitlehookd}{%
\begin{abstract}
In this paper we establish several results on approximate   controllability  of a semilinear wave equation
by making use of a single multiplicative control. These results are  then applied to discuss the exact controllability properties for
the one dimensional  version of the system at hand. The  proof relies on linear semigroup  theory and the results on the additive controllability of the  semilinear wave equation. The approaches are constructive and provide  explicit  steering controls.  Moreover, in the context of undamped wave equation, the exact controllability is established for a time which is uniform for all initial states.
\end{abstract}

{\bf Keywords}: Multiplicative control, semilinear wave equation, approximate controllability, exact controllability.}

\newtheorem{theorem}{Theorem}
\newtheorem{lem}{Lemma}
\newtheorem{remark}{Remark}

\newtheorem{proof}{Proof}

\begin{document}

\maketitle


\section{Introduction}
In this paper, we study the controllability problem for a distributed parameter system governed
by the following  $n-$dimensional wave equation:
\begin{equation}\label{wave-init0}
\left\{%
\begin{array}{lll}
  w_{tt}=&\Delta w + v(x,t)  w  + f(t,w,w_t),&
 \hbox{in}\,\;\;  \Omega\\
w=&0,  & \hbox{on} \ \partial\Omega \\
w(x,0)=&w_{_1},\,  w_{t}(x,0)=w_{_2}, & \hbox{in}\, \Omega \\
\end{array}%
\right.
\end{equation}
where $\Omega$ is a bounded open set of $I\!\!\!R^{n},\; n\ge 1$  with a smooth boundary $\partial\Omega$. The real valued coefficient $v(x,t)$ is the multiplicative control and $f$ is the nonlinearity. Our goal is to identify a set of  states $(w(\cdot,t),w_t(\cdot,t))$ that can be  achieved by system (\ref{wave-init0}) at a time $T>0$ using a suitable control $v(x,t)$. Such problems arise in various real situations
(see \cite{kha0} and the rich references therein).  Research in  the multiplicative  controllability of   distributed  systems   have been the subject of several works. The question of  controllability of PDEs equations by multiplicative  controls has attracted many researchers in the context of  various type of  equations,  such as rod equation \cite{ball,kim95}, Beam equation \cite{bea08}, Schr$\ddot{o}$dinger equation \cite{bea06,kim95,ners},   heat equation \cite{can10,can15,fer12,flor14,kha03,kha0,lin06,ouz16}. Various approaches  were used to tackle the question of multiplicative controllability of hyperbolic equations  like (\ref{wave-init0}). The  homogeneous version of (\ref{wave-init0})   (i.e, $f=0$) has been considered  in \cite{ball,K1,kha04,kha0,ouz14}. The case  of semilinear wave equation has been studied in \cite{kha06} for equilibrium-like states of the form $(y^d_1,0)$ using two controls, i.e. beside the control $v(x,t)$,  a time-dependent control has been considered in the damped part. Furthermore, research in  the  controllability of the semilinear wave equation by additive controls  have been the subject of several works (see \cite{leiv,mari,zha,zua2} and the references therein).\\
In this paper, we study the  approximate and exact controllability for the  system (\ref{wave-init0}) by the means of a single multiplicative control, thus we will have a principal reduction in the means to control the system (\ref{wave-init0}).

The paper is organized as follows: in the second section, we first consider  the question of reaching approximately target states of the form $(w(0),\theta_2)$ by applying a suitable time-independent  control $ v(x,t) = v_T(x)$ at a "short" time $T$.  In the second part of the same section, we  define  a set of target states   $(\theta_1,\theta_2)$ that can be approximately achieved  by using a piecewise static  control in "long" time. In Section 3, we apply the result of Section 2 to define a strategy of the controller $v(x,t)$ in order to get  the exact achievement of a class of target states for both damped and undamped cases.


\section{Approximate  controllability}

\subsection{Preliminaries}
The following lemmas   will be used in several steps  in the proof of our main results.

The next result concerns a Gronwall inequality regarding locally integrable functions.
\begin{lem}\label{Gronw} (see \cite{drag,ye07}).
Let  $\varphi (t)$ be a nonnegative and locally integrable function  on $[0,T],\; 0\le T < +\infty, $ such that  the inequality
$$
\varphi(t) \le a +
b \int_0^{t} \varphi(s)ds, \;\mbox{for} \; a.e. \; t\in [0,T]
$$
holds for some  nonnegative constants $a$ and $b$. Then
$$
\varphi(t) \le a e^{b t},\;\mbox{for} \; a.e. \; t\in [0,T].
$$
\end{lem}

Let us  give the following lemma which concerns the uniform approximation of continuous functions using Bernstein polynomials.
\begin{lem}\label{lemma3}(\cite{dav}, pp. 108-113).
Let $u : [0,1] \to X$ be  a continuous  function from $[0,1]$ to a Banach space $(X,\|\cdot\|_X)$, and let $B_n(u)$  be  the $n$th Bernstein polynomial for $u$:
$$
B_n(u)(t)=\sum_{k=0}^n \left(
                                                                 \begin{array}{c}
                                                                   n \\
                                                                   k \\
                                                                 \end{array}
                                                               \right)  t^k(1-t)^{n-k}  u(\frac{k}{n}),\;\; n\ge 1.$$
Then the  sequence $B_n(u) $ tends uniformly to $u,$ i.e., $\sup_{t\in [0,1]} \|B_n(u)(t)-u(t)\|_X \to 0,$ as $n\to +\infty.$\\
Furthermore, for all $n\ge 1,$ we have:
\begin{equation}\label{deriv-Bern}
B_n(u)'(t)=n\sum_{k=0}^{n-1}  (
                                                                 \begin{array}{c}
                                                                   n-1 \\
                                                                   k \\
                                                                 \end{array}
                                                               ) t^k(1-t)^{n-1-k} (u(\frac{k+1}{n})-u(\frac{k}{n})),
\end{equation}
where  $B_n(u)'(t)$ is the derivative of $B_n(u)(t)$ with respect to $t.$

\end{lem}

Let us  show the following smoothness lemma:
\begin{lem}\label{lemma2}
Let $\Omega$  be an open bounded set of $R^n, \; n\ge 1$. For all $h\in L^\infty(\Omega)$ such that $h\ge 0$, a.e. in $\Omega, $  there exists  $(h_r)\subset C^\infty(R^n)$ such that:

 (i) $(h_r|_\Omega)$ is uniformly bounded with respect to $r$, (where $h_r|_\Omega$ designs the restriction of $h_r$ to $\Omega$),\\
 (ii) for all $r>0; \; h_r>0,\; $ a.e in $\overline{\Omega}, $\\
 and\\
 (iii) $h_r|_\Omega\to h $  in $L^2(\Omega), $ as $r\to 0^+$.
\end{lem}

\begin{proof}
Let us extend $h$  by $0$  to
$R^n$ so that the obtained extension, still denoted by $h$, lies in $L^2(R^n)\cap L^\infty(R^n)$.\\
Let us introduce the following function:
$$
\phi(x)=  \left\{
  \begin{array}{ll}
   c \: e^{\frac{1}{\|x\|^2-1} } , & \hbox{if} \; \|x\|<1 \\
    0, & \hbox{if} \;\|x\|\ge1
  \end{array}
\right.
$$
where $c$ is a positive constant such that: $\int_{R^n}\phi =1.$ For all $r>0$, let $\phi_r(x)=r^{-n} \phi(\frac{x}{r}), \;$ a.e. $ x\in R^n$ and let $k_r$ be the convolution of $h$ with: $\phi_r; \, k_r=\phi_r \ast h.$ This directly yields  $k_r\in {\mathcal C}^\infty(R^n),\, k_r\ge 0$ a.e. in $\overline{\Omega} $ and $ k_r\to h$  in $L^2(\Omega), $ as $r\to 0^+$ (see \cite{brez}, pp. 69-71). Moreover, for every $r>0$ and for a.e. $x\in \Omega$, we have:
$$
\begin{array}{cccc}
  k_r(x)  & = & c\: r^{-n} \displaystyle\int_{B(x,r)}  h(s) e^{\frac{1}{\|\frac{x-s}{r}\|^2-1}}  ds \\
  \\
   & \le & \displaystyle c \: \|h\|_{L^\infty(R^n)} \: \displaystyle\int_{B(O,1)} ds.
 \end{array}
$$
In other word, the sequence  $(k_r)$ is uniformly bounded with respect to $r$.
We conclude that $h_r := k_r+\frac{r}{r+1}, r >0$ satisfies the claimed properties.
\end{proof}

\subsection{A partial approximate controllability result}
Let us consider the system (\ref{wave-init0}) evolving on a time-interval $(0,T_0)$ with a nonlinear term $f : (0,T_0)\times H_0^1(\Omega) \times L^2(\Omega)   \to L^2 (\Omega) $ which is  globally Lipschitz. \\
Letting $y=(w,w_t)\in H:=H_0^1(\Omega)\times L^2(\Omega), $ we obtain the following equivalent first order system:
\begin{equation}\label{Eq}
\left\{\begin{array}{llll}
  y_{t}&=&  A y + v(t) B y+F(t,y), & t\in (0,T_0)\\
  \\
  y(0)&=&y_0=(w_1,w_2) &
\end{array}
\right.
\end{equation}
where $v(t)=v(\cdot,t)\in \mathcal{U}:=L^\infty(\Omega), \; B=\left(
    \begin{array}{cc}
      0 & 0 \\
      I & 0 \\
    \end{array}
  \right) $
and
$A=\left(
    \begin{array}{cc}
      0 & I \\
      \Delta & 0 \\
    \end{array}
  \right) $ with domain $ \mathcal{D}(A)=\big (H_0^1(\Omega)\cap
H^2(\Omega)\big ) \times H_0^1(\Omega)
$
and where for all $t\in(0,T_0)$ and $y=(w_1,w_2)\in H;\;  F(t,y)=(0,f(t,w_1,w_2)).$ Here, the state space $H$ is endowed with the following inner product: $\langle (u_1,u_2),(v_1,v_2)\rangle= \langle u_1, v_1\rangle_{H_0^1(\Omega)}+\langle u_2, v_2\rangle_{L^2(\Omega)}$ with  corresponding norm $\|\cdot\|.$ With this Hilbert structure, the operator $A$ generates a  semigroup of isometries $S(t)$.\\
For any $\xi\in L^2(\Omega)$ we set: $\Lambda(\xi):=\{x\in\Omega/\; \xi(x)\ne 0\}$ and $ {\bf 1}_{\Lambda(\xi)}$ will denote the characteristic function of $\Lambda(\xi).$

Our first main result  concerns the  approximate controllability toward a target state $w(T) = w_1$, $\partial_t w(T) = \theta_2$ within an arbitrarily small time-interval $(0, T)$, which depends on the choice of the initial state $y_0=(w_1,w_2)$,  the target state $y^d=(w_1,\theta_2)$ and the  precision of steering. The main idea here consists on looking for a static control such that the respective solution to (\ref{Eq}) is such that $y(T)-y^d\to 0, $ as $T\to 0^+.$ This idea was first used  by Khapalov in \cite{kha03} in the context of reaction-diffusion equation (see also   \cite{can15}).

\begin{theorem}\label{thm1c}
 Let $(w_1,w_2)\in H$ and $ \theta_2\in L^2(\Omega)$  and let us set $ a(x):=\frac{\theta_2-w_2}{w_1}{\bf 1}_{\Lambda(w_1)}.$  Assume that: (i) $ a \in L^{\infty}(\Omega) $ and
(ii) for a.e., $x\in \Omega;\; w_1(x)=0 \Rightarrow \theta_2(x)=w_2(x).$ Then for any $\epsilon>0, $ there are a time $ T=T(w_1,w_2,\theta_2,\epsilon) \in (0,T_0)$ and a static control $v(\cdot,t)=v_T(\cdot)\in  W^{2,\infty}(\Omega)$ such that for the respective solution to (\ref{wave-init0}), the following inequalities hold:
$$
\|w(T)-w_1\|_{H_0^1(\Omega)} <\epsilon \;\; \mbox{and}\;\; \|w_t(T)-\theta_2\|_{L^2(\Omega)} <\epsilon.
$$
\end{theorem}
\begin{proof}
Let $\epsilon>0$, and let us consider the state $y^d=(w_1,\theta_2)$  to be achieved. For any time of steering $0<T<T_0$ we consider the  control
\begin{equation}\label{contr1}
  v(x,t)=v_T(x):=\frac{a(x)}{T},\; t\in (0,T_0).
 \end{equation}
Since $a\in L^{\infty}(\Omega), $ there is a unique mild solution $y(t)$ to (\ref{Eq}) (see \cite{pazy},  p. 184),
 which  is given by the following  variation of constants formula:
\begin{equation}\label{vcf2}
y(t)= S(t)y_0+ \int_0^t S(t-s) \big ( v_T(x)  B y(s) + F(s,y(s)) \big ) ds,
\end{equation}
for all $ t$ in $ [0,T_0].$ We aim to show that the control (\ref{contr1}) guarantees the steering of system (\ref{Eq}) to $y^d$ at any  small time  $T>0$, so we can assume in the sequel that $0<T<T_0:=1$.

{\bf Case 1.} $a(\cdot) \in W^{2,\infty}(\Omega)$ and $y_0\in \mathcal{D}(A).$\\
We will distinguish two subcases:

{\bf Case 1.1.} Assume that the operator $F$ is $\mathcal{C}^1$ and globally Lipschitz from $(0,T_0)\times H$ to $H$. Here, the mild solution is a classical one. In particular we have $ y (t)\in \mathcal{D}(A)=H_0^1(\Omega)\cap H^2(\Omega) \times H_0^1(\Omega),\; \forall t\in [0,T_0] $ (see \cite{pazy}, p. 187).  \\
It comes from the assumption (i) and from (\ref{contr1}) that: $
e^{Tv_TB} =\left(
           \begin{array}{cc}
             I & 0 \\
            a & I \\
           \end{array}
         \right),$
so  the assumption (ii) leads to: $e^{Tv_T B}y_0=y^d. $\\
The idea of the proof will consist on proving the following  formula:
\begin{equation}\label{Avcf}
y(T)-y^d=\int_0^Te^{(T-s)v_T(x)B} \big ( A y(s) + F(s,y(s)) \big ) ds,
\end{equation}
and   showing that the term in the right-hand side of the  relation (\ref{Avcf}) tends to zero as $T\to 0^+.$\\
In order for $y(t)$ to satisfy  (\ref{Avcf}), it suffices to show that $Ay(\cdot) \in L^1(0,1)$  (see \cite{bal78}). For this end, let us apply the bounded operator $A_\lambda = \lambda R(\lambda;A) A$ to (\ref{vcf2}), where $R(\lambda;A)$ is the resolvent of $A$. Thus
  $$  A_\lambda y(t)  = S(t)A_\lambda y_0+ \displaystyle \int_0^t A_\lambda S(t-s) (v_T(x) B y(s) + F(s,y(s))) ds.
  $$
 This gives
\begin{equation}\label{6*}
\begin{array}{lll}
A_\lambda y(t) &=&  S(t)A_\lambda y_0 +   \int_0^t A_\lambda S(t-s)  (v_T(x) B  y(s)) ds + \\
&&
 \int_0^t \lambda R(\lambda,A) S'(t-s) F(s,y(s)) ds
\end{array}
\end{equation}
where $S'(t)$ is the derivative of $S(t)$ with respect to $t$.
We have
$$
\begin{array}{lll}
  \int_0^t \lambda R(\lambda,A) S'(t-s) F(s,y(s)) ds   &=& -
\int_0^t \frac{d}{ds} \big ( \lambda R(\lambda;A) S(t-s)  F(s,y(s)) \big ) ds - \\
&&\int_0^t \lambda R(\lambda;A) S(t-s) \big (\frac{\partial F}{\partial s}(s,y(s))+
\frac{\partial F}{\partial y}(s,y(s)) y_s(s) \big ) ds,
\end{array}
$$
where $y_s$ refers to the derivative w.r.t "$s$".
Thus
$$
\begin{array}{lll}
  \int_0^t \lambda R(\lambda,A) S'(t-s) F(s,y(s)) ds   &=&  \lambda R(\lambda;A) \big ( F(t,y(t)) -S(t) F(0,y_0)\big ) - \\
   &&\int_0^t \lambda R(\lambda;A) S(t-s) \big (\frac{\partial F}{\partial s}(s,y(s)) + \frac{\partial F}{\partial y}(s,y(s)) y_s(s) \big ) ds.
\end{array}
$$
Using the fact that $S(t)$ is a contraction semigroup,  we deduce that:
$$
\begin{array}{lll}
\|A_\lambda y(t)\| &\le & \|Ay_0\| +   \int_0^t  \|   A (v_T(x) B y(s)) \| ds + \|F(0,y_0)\| +\| F(t,y(t))\| +\\
\\
&&\int_0^t \| \frac{\partial F}{\partial s}(s,y(s)) + \frac{\partial F}{\partial y}(s,y(s)) y_s(s)\| ds\cdot
\end{array}
$$
Moreover, using  (\ref{vcf2}), we deduce via Gronwall's inequality that:
\begin{equation}\label{y-y0}
\|y(t)\|\le C \|y_0\|+C,\; \forall t\in[0,T]\subset [0,1],
\end{equation}
for some positive constant $C=C(\|a\|_{L^{\infty}(\Omega)})>0$ which is independent of $T$.\\
Then using the fact that $F$ is Lipschitz we get:
\begin{equation}\label{F(y)-y0}
 \| F(t,y(t))\|\le C(1+\|y_0\|),\; C=C(\|a\|_{L^{\infty}(\Omega)})>0,
\end{equation}
and
$$
\int_0^t \| \frac{\partial F}{\partial s}(s,y(s)) + \frac{\partial F}{\partial y}(s,y(s)) y_s(s)\| ds\le LT + L\int_0^t  \| y_s(s) \| ds.
$$
Then,  letting $\lambda\to +\infty,$  we deduce that:
\begin{equation}\label{5*}
\begin{array}{lll}
\|A y(t)\| &\le&  \|Ay_0\| +    \int_0^t  \| A  (v_T(x) B y(s)) \| ds+ \\
&&
C(1+\|y_0\|_{\mathcal{D}(A)}) + L\int_0^t  \| y_s(s) \| ds
\end{array}
\end{equation}
where $L$ is a Lipschitz constant of $F$ and the constant $C=C(\|a\|_{L^{\infty}(\Omega)})>0$  is independent of $T$.\\
In the sequel, the  letter $C$ will be used to denote a generic positive constant which is independent of $T$.\\
Let us now study the terms of the right hand of inequality (\ref{5*}). We have $v_T(x)By(t)=(0,\frac{a(x)}{T}w(t)),$ thus  since $a\in W^{2,\infty}(\Omega) $ it comes that $v_T(x)By(t)\in \mathcal{D}(A)$ for all $t\in [0,T].$ Moreover, we have the following  second order Leibniz rule:
$$\Delta (a w))=\Delta (a) w +2 \nabla (a)\cdot \nabla(w) + a \Delta(w),\; \forall w\in H^2(\Omega), $$
from which we get:
$$
\|A(v_T(x) By(s))\|= \frac{1}{T}\|\Delta (a w(s))\|_{L^2(\Omega)}\le \frac{C}{T} \|y(s)\|_{\mathcal{D}(A)},\; \forall s\in [0,T]$$
where $C=C(\|a\|_{W^{2,\infty}(\Omega)})$ is independent of $T.$ It follows that:
\begin{equation}\label{GP}
\int_0^t \|A  (v_T(x) B y(s) )\| ds \le \frac{C}{T} \int_0^t\|y(s)\|_{\mathcal{D}(A)} ds,\, \forall t\in [0,T]\cdot
\end{equation}
Since $y(t)$ is a classical solution, we have
\begin{equation}\label{y'}
\begin{array}{ccc}
\|y_t(t)\|&=&\|Ay(t)+v_T(x) B y(t)+F(t,y(t))\|\\
&\le &\frac{C}{T} \|y(t)\|_{\mathcal{D}(A)}+T+\|F(0,0)\|
\end{array}\end{equation}
for all $0< t\le T,$ where $C=C(\|a\|_{L^{\infty}(\Omega)})$.\\
Reporting  (\ref{GP}) and (\ref{y'})  in (\ref{5*}) and taking into account (\ref{y-y0}),  we deduce via  Gronwall's inequality that:
\begin{equation}\label{ynDA}
\|y(t)\|_{D(A)}\le C \|y_0\|_{\mathcal{D}(A)}+C,\; \forall t\in [0,T],
\end{equation}
where $C=C(\|a\|_{W^{2,\infty}(\Omega)})$  is independent of $T.$ Thus $Ay(\cdot) \in L^1(0,T)$, and hence the following variation of constants formula holds:
\begin{equation}\label{y-vcf}
y(t)=  e^{tv_T(x)B} y_0+\int_0^te^{(t-s)v_T(x)B} \big (A y(s)+ F(s,y(s)) \big ) ds,\; \forall t\in [0,T],
\end{equation}
from which it comes
\begin{equation}\label{vcf1}
y(T)-y^d=\int_0^Te^{(T-s)v_T(x) B} \big(A y(s)+F(s,y(s))\big)  ds.
\end{equation}
Based on (\ref{vcf1}) and using (\ref{ynDA}) and the fact that $F$ is Lipschitz, we deduce that:
\begin{equation}\label{appro-est}
\|y(T)-y^d\|\le C_*  T  (\|y_0\|_{\mathcal{D}(A)}+1),\; C_*=C_*(\|a\|_{W^{2,\infty}(\Omega)}),
\end{equation}
and hence $\|y(T)-y^d\| <\epsilon$, whenever $0<T<\inf ( {T_0=1, \frac{\epsilon}{C_* (1 +  \|y_0\|_{\mathcal{D}(A)})}})$.

{\bf Case 1.2. } Here, we only assume that the operator $F$ is globally Lipschitz  from $(0,T_0)\times H$ to $H$ (with a Lipschitz constant $L>0$), and let $y(t)$ be the mild solution of (\ref{Eq}) corresponding to control $v_T(x)$ given by (\ref{contr1}). Then  we can    approximate the function $t\to F(t,y(t))$ uniformly  with  $\mathcal{C}^1-$functions $(F_n)$  in $[0,T_0]=[0,1]$.  More precisely, according to Lemma \ref{lemma3}, one can  consider the following Bernstein polynomial:
$$
F_n(t)=\sum_{k=0}^n \left(
                                                                 \begin{array}{c}
                                                                   n \\
                                                                   k \\
                                                                 \end{array}
                                                               \right)  t^k(1-t)^{n-k}  F(\frac{k}{n},y(\frac{k}{n})),\;t\in [0,1]\; n\ge 1.$$
From (\ref{F(y)-y0}) we get: $$\sup_{t\in [0,T]}  \|F_n(t)\|\le C(1+ \|y_0\|),\; C=C(\|a\|_{L^\infty(\Omega)})>0.$$
Moreover,  for all $n\ge 1$ we have:
\begin{equation}\label{deriv-Bern}
F'_n(t)=n\sum_{k=0}^{n-1}  \left(
                                                                 \begin{array}{c}
                                                                   n-1 \\
                                                                   k \\
                                                                 \end{array}
                                                               \right) t^k(1-t)^{n-1-k} \big ( F(\frac{k+1}{n},y(\frac{k+1}{n}))-F(\frac{k}{n},y(\frac{k}{n})) \big ),
\end{equation}
where  $F'_n(t)$ is the derivative of $F_n(t)$.\\
Let us  show that the sequence of derivative $(F'_n)$ is uniformly bounded in $[0,T]$. \\
For all $h, t\in [0,T]$ such that $t+h\in [0,T],$ we have:
$$
\begin{array}{lll}
y(t+h)-y(t)&=&\int_0^h S(t+h-s) \big \{ \frac{a}{T} B y(s) +F(s,y(s)) \big \} ds +S(t+h)y_0-\\
&& S(t)y_0+\int_0^t S(t-s) \big \{ \frac{a}{T} B( y(s+h)-y(s) )  \big \}ds +\\
&&\int_0^t S(t-s) \big \{  F(s+h,y(s+h))-F(s,y(s)) \big \}ds,
\end{array}
$$
from which, we derive:
$$
\|y(t+h)-y(t)\| \le h \|Ay_0\| +h \frac{C(1+\|y_0\|)}{T}  +  \int_0^t \big \{ L h+ (\frac{\|a\|_{\infty}}{T}+L) \|y(s+h)-y(s)\| \big \} ds,
$$
where $C=C(\|a\|_{L^\infty (\Omega)})$ is independent of $T$, which by Gronwall's inequality gives the following estimate:
$$
\|y(t+h)-y(t)\|\le \frac{C(1+ \|y_0\|_{{\mathcal D}(A)})}{T} h,
$$
where $ C=C(\|a\|_{L^\infty(\Omega)})$ is independent of $T$.\\
It follows from the expression of $F'$ and the last inequality that:
$$
\displaystyle\sup_{n\ge 1}\sup_{t\in [0,T]}  \|F'_n(t)\| \le L_T := \frac{C(1+\|y_0\|_{\mathcal{D}(A)})}{T},\; C=C(\|a\|_{L^\infty(\Omega)})\cdot $$
As a consequence, $F_n$ is $L_T-$Lipschitz on $[0,T]$.\\
In the sequel, we will apply the techniques of Case 1.1 to the following approached system:
 \begin{equation}\label{Eqn}
\left\{\begin{array}{ll}
 \frac{d}{dt} y_{n} (t)=  A y_n(t) + v_T(x) B y_n(t)+F_n(t)& \\
 \\
  y_n(0)=y_0=y(0)&
\end{array}
\right.
\end{equation}
Let  $y_n(t)$ denote the classical solution of the  system (\ref{Eqn}).\\
Based on the variation of constants formula, we can show via the Gronwall's inequality that there is  $N=N(T,\epsilon)\in \mathbf{I\!N}$  such that:
$$\|y_{N}(T)-y(T)\|< \epsilon/2.$$
Moreover, applying the relation (\ref{6*}) to $y_n(t)$ leads to:
\begin{equation}\label{6*n}
\begin{array}{lll}
A_\lambda y_n(t) &=&  S(t)A_\lambda y_0 +   \int_0^t A_\lambda S(t-s)  (v_T(x) B  y_n(s)) ds+ \\
&&\int_0^t \lambda R(\lambda,A) S'(t-s) F_n(s) ds\cdot
\end{array}
\end{equation}
We have
$$
\begin{array}{lll}
  \int_0^t \lambda R(\lambda,A) S'(t-s) F_n(s) ds   &=& -
\int_0^t \frac{d}{ds} \big ( \lambda R(\lambda;A) S(t-s)  F_n(s) \big ) ds +
 \int_0^t \lambda R(\lambda;A) S(t-s) F_n'(s)  ds\\
&=&
\lambda R(\lambda;A) \big ( S(t) F_n(0)-F_n(t)\big ) + \int_0^t \lambda R(\lambda;A) S(t-s) F_n'(s).
\end{array}
$$
Then  we deduce that:
$$
\begin{array}{lll}
\|A_\lambda y_n(t)\| &\le&  \|Ay_0\| +   \int_0^t  \|   A (v_T(x) B y_n(s)) \| ds +\\
& &\|F_n(0)\| +\| F_n(t)\| + \int_0^t \| F'_n(s)\| ds
\end{array}
$$
Letting $\lambda\to +\infty$, we get
\begin{equation}\label{5*n}
\|A y_n(t)\| \le  \|Ay_0\| +    \int_0^t  \| A  (v_T(x) B y_n(s)) \| ds + C \|y_0\|+C.
\end{equation}
where $C=C(\|a\|_{L^\infty(\Omega)})$ is a positive constant which is independent of $T.$
Then by proceeding as in  the Case 1.1., we get an estimate like (\ref{appro-est}), namely:
 \begin{equation}\label{nappro-est}
\|y_N(T)-y^d\|\le C  T  (\|y_0\|_{\mathcal{D}(A)}+1),
\end{equation}
where $ C=C(\|a\|_{W^{2,\infty}(\Omega)})>0$  is independent of $N.$\\
It follows  that $\|y_{N}(T)-y^d\| < \epsilon/2, $ for some $T$ small enough, and hence
$$
\|y(T)-y^d\|< \epsilon.$$

{\bf Case 2. } $a(\cdot)\in W^{2,\infty}(\Omega)$ and $y_0\in H.$\\
 Let $T>0,$ and for all $\lambda>0$ we set  $y_{0\lambda}:=\lambda R(\lambda;A)y_{0}\in \mathcal{D}(A).$ Let $y_{\lambda}$ be the mild solution of  (\ref{Eq}) corresponding to the initial state  $y_{0\lambda}=(w_{1\lambda},w_{2\lambda})$   with  the same  control as in the Case 1., i.e.,
$v(x,t)=v_T(x)=\frac{a}{T},\; 0< t < T_0$.\\
We have
$$
\|y(T)-y^d\|\le \|y(T)-y_{\lambda}(T)\| + \|y_{\lambda}(T)-e^{aB}  y_{0\lambda}\|+\|e^{aB}  y_{0\lambda} - y^d\|.
$$
It follows from the variation of constants formula that:
$$
\begin{array}{lll}
y_{\lambda}(t)-y(t)&=&\int_0^t S(t-s) \big ( v_T(x) B (y_{\lambda}(s)-y(s))  \big )ds+S(t)y_{0\lambda}-
\\
&&S(t)y_0+\int_0^t S(t-s) \big (  F(s,y_{\lambda}(s))-F(s,y(s)) \big )ds\cdot
\end{array}
$$
Then, using  the contraction property of the semigroup $S(t)$, it comes:
$$
\begin{array}{lll}
\|{y}_{\lambda}(t)-y(t)\| &\le& \|{y}_{0\lambda}-y_0\|+\frac{\|a\|_{L^\infty(\Omega)}}{T} \int_0^T \|{y}_{\lambda}(s)-y(s)\| ds +\\
&&L\int_0^T \|{y}_{\lambda}(s)-y(s)\| ds, \; \forall t\in [0,T].
\end{array}$$
 Gronwall's lemma yields
$$
\|{y}_{\lambda}(T)-y(T)\|\le C \|{y}_{0\lambda}-y_0\|, \; (C=C(\|a\|_{L^\infty(\Omega)})).
$$
It follows from $y^d= e^{aB}  y_{0}$ that:
$$
\|e^{aB}  y_{0\lambda} - y^d\|\le \|a\|_{L^\infty(\Omega)} \|{y}_{0\lambda}-y_0\|.
$$
We deduce that there is a $\lambda>0$, which is independent of $T\in (0,1), $ such that:
$$
\|{y}_{\lambda}(T)-y(T)\| + \|e^{aB}  y_{0\lambda} - y^d\|<\frac{\epsilon}{2}.
$$
For such a $\lambda,$ we deduce from the same arguments as  in the Case 1 that there exists $0<T<1$ such that:
$
\|y_{\lambda}(T)-e^{aB}  y_{0\lambda} \| <\frac{\epsilon}{2}.
$\\ We conclude that: $$
\|y(T)-y^d\| <\epsilon.$$

{\bf Case 3: } $a(\cdot)\in L^\infty(\Omega)$ and $y_0\in H.$\\
From Lemma \ref{lemma2}, there is a sequence  $(a_k)  \subset W^{2,\infty}(\Omega)$ which is uniformly  bounded  on $\Omega$ such that $a_k \to a$  in $ L^2(\Omega), $ as $k \to +\infty. $ Here, we will  consider the control: $v_T(x)=\displaystyle\frac{a_k}{T }$ for a suitably selected (large enough) $k\in \mathbf{N}$, and let $y(t)$ be the  corresponding solution to (\ref{Eq})  with the  initial state $y(0)=y_0=(w_1,w_2).$\\
Now, let $(w_{2l})\in  L^\infty(\Omega)$ be such that $w_{2l} \to w_2$ in $L^2(\Omega), $ as $l\to +\infty,$ and let us consider the initial state $y_{0l}=(w_1,w_{2l}).$ \\
We have the following triangular inequality:
$$
\begin{array}{ccc}
\|y(T)- e^{aB} y_{0}\| &\le& \|y(T)-e^{a_kB} y_0\|  + \|e^{a_kB} y_0- e^{a_kB} y_{0l}\|  + \\
&&\|e^{a_kB} y_{0l}-e^{aB} y_{0l}\|   + \|e^{aB} y_{0l}-e^{aB} y_{0}\|
\end{array}
$$
From the relation  $e^{a_kB} =\left(
           \begin{array}{cc}
             I & 0 \\
           a_k & I \\
           \end{array}
         \right),$ we deduce that:
         $$
\|e^{a_kB} y_0- e^{a_kB} y_{0l}\|  + \|e^{aB} y_{0l}-e^{aB} y_{0}\|  \le \sup_{k\in I\!\!N} \big ( 1,\|a_k\|_{L^\infty (\Omega)},\|a\|_{L^\infty (\Omega)} \big ) \,  \|y_{0l}-y_0\|$$
and $
e^{a_kB} y_{0l}-e^{aB} y_{0l}=(0,(a_k-a) w_{2l}).$ Let $l\in \mathbf{N}$ be such that
$$
\sup_{k\in \mathbf{N}} \big ( 1,\|a_k\|_{L^\infty (\Omega)},\|a\|_{L^\infty (\Omega)} \big ) \,   \|w_{2l}-w_2\|_{L^2(\Omega)}<\frac{\epsilon}{3},
$$
and for such value of $l,$ we consider a $k$  such that
$$
\|a_k -a\|_{L^2(\Omega)} \| w_{2l}\|_{L^\infty(\Omega)} <  \frac{\epsilon}{3}.
$$
Then, for this value of $k$, it comes from the  Case 2 that there exists $T>0$ such that:
$$
\|y(T)-e^{a_kB} y_0\|  < \frac{\epsilon}{3}.
$$
We conclude that
$$
\|y(T)-e^{aB} y_{0}\| < \epsilon.
$$
Finally, since $e^{aB} y_{0}=y_d, $ it comes
$$
\|w(T)-w_1\|_{H_0^1(\Omega)} <\epsilon \;\; \mbox{and}\;\; \|w_t(T)-\theta_2\|_{L^2(\Omega)} <\epsilon.
$$
\end{proof}
\begin{remark}
For any initial state $(w_1,w_2), $ the set of reachable states $\theta_2$ identified in the above theorem is  convex.
\end{remark}
\subsection{Global approximate controllability}
In this subsection, we will consider the following equation:
\begin{equation}\label{wave-init}
\left\{%
\begin{array}{llll}
  w_{tt}&=&\Delta w + v(x,t)  w -h(x) w_t + f(w),
& \hbox{in}\:\;\;\;
 \Omega \times (0,T)\\
w&=&0,   & \hbox{on} \ \partial\Omega \times (0,T)\\
w(x,0)&=&w_{_1},\,  w_{t}(x,0)=w_{_2}, & \hbox{in}\,\ \Omega
\end{array}%
\right.
\end{equation}
where $T>0,\; h\in L^\infty(\Omega)$ and the nonlinear term $f :  L^2(\Omega)\to L^2 (\Omega) $ is a   globally Lipschitz function. Here,  we will study the approximate controllability problem for the system (\ref{wave-init}) toward a full state $(\theta_1,\theta_2) $ by using two static controls, applied subsequently in time.\\
For any $\zeta\in H^2(\Omega)$ we set $b_\zeta:=
-\displaystyle\frac{ \Delta\zeta + f(\zeta)}{\zeta}{\bf 1}_{\Lambda(\zeta)}, $
and let us consider the following assumptions:\\
$(\mathcal{P}_1)$: $b_\zeta\in L^\infty(\Omega),$\\
$(\mathcal{P}_2)$:  $ h\ge 0, \; a.e.\; \Omega$.
 and there exist $ \delta , \mathcal{T} >0 $ such that:
  \begin{equation}\label{coe}
\int_0^{\mathcal{T}} \int_\Omega h(x) |\varphi_t(x)|^2 dx  dt \ge \delta
\|(\varphi_1,\varphi_2)\|^2_{H},\; \; \forall
(\varphi_1,\varphi_2)\in H,
\end{equation}
where $\varphi$ is the solution of
\begin{equation}\label{dua}
\varphi_{tt}=\Delta \varphi+b_\zeta(x)\varphi+f(\varphi),\;\; \varphi(0)=\varphi_1\in
H_0^1(\Omega),\;\; \varphi_t(0)=\varphi_2\in L^2(\Omega),
\end{equation}
$(\mathcal{P}_3)$:  $\langle f(y)+b_\zeta(x) y,y\rangle_{L^2(\Omega)} \le 0,\; \forall y\in L^2(\Omega), $\\
$(\mathcal{P}_4)$: for a.e. $x\in \Omega,$ we have: $\zeta(x)=0 \Rightarrow f(\zeta)(x)=0.$\\

We have the following remarks regarding the  estimate  (\ref{coe}).
\begin{remark}\label{rem-obs}
\begin{enumerate}
  \item For $f=b_\zeta=0,$ the inequality (\ref{coe}) was established
for $ \mathcal{T}$ large enough provided there is a subset $O$ of the support of $h$ satisfying the following so-called
geometrical control condition (GCC): "there exists $x_0 \in  \mathbf{R}^n$ such that $O$ is a
neighborhood of the closure of the set $\Gamma (x_0) :=\{x\in \partial\Omega
/\; (x - x_0)\nu(x)>0\}$", where $\nu(x)$ denotes the unit outward
normal at $x\in \partial\Omega$ (see \cite{bar}).
In particular, for $n=1$ and  $g={\bf 1}_\omega$ the estimate  (\ref{coe})  holds   for  $\omega = (a,b)\subset \Omega=(0,1)$ and $\mathcal{T}>2\inf(a,1-b)$ (see \cite{zua2}).
 \item Using  robustness results on the observability property (see \cite{ouz14}), we can see that  (\ref{coe}) also holds under the  geometrical control condition for small Lipschitz constant $L_{\zeta}$ of the operator:  $y \in L^2(\Omega) \mapsto f(y)+b_{\zeta}(x)  y$. Indeed, let the above (GCC) hold, so that:
\begin{equation}\label{coe0}
\int_0^{\mathcal{T}} \int_\Omega h(x) |\phi_t(x)|^2 dx  dt \ge \delta
\|(\phi_1,\phi_2)\|^2_{H},\; \; \forall
(\varphi_1,\varphi_2)\in H,
\end{equation}
(for some $\mathcal{T}, \delta>0$), where $\phi$ is the solution of
$ \phi_{tt}=\Delta \phi, \phi(0)=\phi_1\in
H_0^1(\Omega), \phi_t(0)=\phi_2\in L^2(\Omega).$
We can easily show that the solution $T(t)y_0=(\varphi(t),\varphi_t(t)),\;  y_0=(\varphi_1,\varphi_2)$ of (\ref{dua}) verifies
$\|T(t)y_0\|\le e^{L_{\zeta} t} \|y_0\|,\; \forall t\ge 0.$ Then using the variation of constants formula, we get:
$$
|\langle CS(t)y_0,S(t)y_0\rangle|\le \|h\|_{L^\infty(\Omega)} \mathcal{T}  L_{\zeta} (1+e^{L_{\zeta} \mathcal{T}})  e^{L_{\zeta} \mathcal{T}}\|y_0\|^2  + |\langle CT(t)y_0,T(t)y_0\rangle|
$$
where $C= \left(
                     \begin{array}{cc}
                       0 & 0 \\
                       0 & h(x) I \\
                     \end{array}
                   \right). $
From this and (\ref{coe0}) it comes:
$$
\int_0^{\mathcal{T}}|\langle CT(t)y_0,T(t)y_0\rangle| dt \ge (\delta-\beta)\|y_0\|^2 \; \;\mbox{with}\; \;\beta=\|h\|_\infty \mathcal{T}^2 e^{L_{\zeta}\mathcal{T}} L_{\zeta} (1+e^{L_{\zeta} \mathcal{T}}).$$
Hence the estimate (\ref{coe}) holds whenever $L_{\zeta}<\gamma^{-1}(\delta), $ where $\gamma^{-1}$ is the inverse function of $\gamma: s\mapsto \|h\|_\infty \mathcal{T}^2 s e^{s\mathcal{T}}  (1+e^{s \mathcal{T}} ).$

 \item An other situation in which (\ref{coe}) holds  is the case of functions: $f(y )(x)=k(y(x)), $ where $k: R\to R$ is   such that $k(0)=0$ and  $s k (s)\le - s^2\|b_\zeta\|_{L^\infty(\Omega)},\; \forall s\in R $ (see \cite{teb}).
\end{enumerate}
\end{remark}

The following result concerns the  approximate controllability  toward  target states of the form $(\theta_1, 0)$.

\begin{theorem}\label{thm2-i}
Let  $ \theta_1\in    H_0^1(\Omega)\cap H^2(\Omega) $ and let assumptions $(\mathcal{P}_1)-(\mathcal{P}_4)$ hold for $\zeta=\theta_1$. Then for every  initial state $(w_0,w_1)\in H$ and for every $\epsilon > 0,$  there are a time $T =T(w_1,w_2,\theta_1,\epsilon)>0$ and  a piecewise static  control
$v(\cdot,t)=v_T(\cdot)\in L^\infty(\Omega)$  such that the respective solution  to (\ref{wave-init}) satisfies:
\begin{equation}\label{appro-i}
\|w(T)-\theta_1\|_{H_0^1(\Omega)}+\|w_t(T)\|_{L^2(\Omega)} <\epsilon\cdot
\end{equation}
\end{theorem}
\begin{proof}
Let   $\epsilon>0 $ be fixed.  Let us consider the control:
\begin{equation}\label{1v(x,t)}
v(x,t)=v_T(x)=-\frac{\Delta \theta_1+f(\theta_1)}{ \theta_1 } {\bf 1}_{\Lambda(\theta_1)},
\end{equation}
and let  us set  $z=w-\theta_1$. Thus the system (\ref{wave-init}) takes the form:
 \begin{equation}\label{1wave0}
\left\{%
\begin{array}{llll}
  z_{tt}&=&\Delta z + v(x,t)(z+\theta_1) +\Delta \theta_1- h(x) z_t+
  f(z+\theta_1),
& \hbox{in}\:\;\;\;\;
\Omega\times (0,T)\\
z&=&0, \;  & \hbox{on} \, \ \partial\Omega \times (0,T)\\
z(0)&=&z_1,\; z_t(0)=z_2, & \hbox{in}\,\ \Omega \\
\end{array}%
\right.
\end{equation}
where $z_1=w_1-\theta_1$ and $z_2=w_2.$\\
We have: $\Delta\theta_1  -(\Delta \theta_1/\theta_1) {\bf 1}_{\Lambda(\theta_1)} \theta_1=0$ (see  \cite{att,ouz14}). Then, using this and  the fact that for almost every $x$ in $\Omega;\ : \theta_1(x)=0 \Rightarrow f(\theta_1)(x)=0,$ the system (\ref{1wave0}) (controlled with (\ref{1v(x,t)})) becomes:
 \begin{equation}\label{1wave-aux}
\left\{%
\begin{array}{llll}
z_{tt}&=&\Delta z +b_{\theta_1}(x)z- h(x) z_t+f_{\theta_1}(z),
& \hbox{in}\;\;\;\;
\Omega\times (0,T)\\
z&=&0, \;  & \hbox{on} \ \partial\Omega \times (0,T)\\
z(0)&=&z_1,\; z_t(0)=z_2, & \hbox{in}\,\ \Omega \\
\end{array}%
\right.
\end{equation}
where $f_{\theta_1}(z)=f(z+{\theta_1})-f({\theta_1}),\; z\in L^2(\Omega)$.\\
Let $\vartheta$ be a solution of the system:
\begin{equation}\label{1duatheta}
\vartheta_{tt}=\Delta \vartheta+b_{\theta_1}(x)\vartheta+f_{\theta_1}(\vartheta),\;\; \vartheta(0)=\vartheta_1\in
H_0^1(\Omega),\;\; \vartheta_t(0)=\vartheta_2\in L^2(\Omega).
\end{equation}
Then, remarking that   $\varphi:=\vartheta+\theta_1$ satisfies the equation: $\varphi_{tt}=\Delta \varphi+b_{\theta_1}(x)\varphi+f(\varphi),$ we deduce by taking $\varphi(0)=\vartheta_1$ and $\varphi_t(0)=\vartheta_2$ in (\ref{dua}) that:
\begin{equation}\label{1coetheta}
\int_0^{\mathcal{T}} \int_\Omega h(x) \vartheta_t(x)^2 dx  dt \ge \delta
\|(\vartheta_0,\vartheta_1)\|^2_{H},\; \; \forall
(\vartheta_1,\vartheta_2)\in H.
\end{equation}
Moreover, since $f_{\theta_1}$ is Lipschitz and satisfies  $f_{\theta_1}(0)=0$ and $\langle f_{\theta_1}(z)+b_{\theta_1}(x)z,z\rangle\le 0, $ for all $ z\in L^2(\Omega), $ we deduce that the solution of (\ref{1wave-aux}) can be defined for all $t\ge 0$  and satisfies the following exponential decay (see \cite{teb}):
$$\|(z(t),z_t(t))\| \le M e^{-\sigma t} \|(z(0),z_t(0))\|,\; \forall t\ge 0,
$$
for some constants  $M, \sigma>0$ which are independent of $T$.\\
We deduce that   for  $T>\frac{1}{\sigma} | \ln(\frac{M\|(w_1-\theta_1,w_2)\|}{\epsilon}) |$, the solution of
(\ref{wave-init}) satisfies the following estimate:
\begin{equation}\label{1estim-thetea-0}
\|w(T)-\theta_1\|_{H_0^1(\Omega)}+\|w_t(T)\|_{L^2(\Omega)} <\epsilon.
\end{equation}
\end{proof}
Our  main result in this section concerns the case of a full target state $\theta=(\theta_1,\theta_2)$ and is  stated as follows:
\begin{theorem}\label{thm2}
 Let  $ (\theta_1,\theta_2)\in   H_0^{1}(\Omega) \cap H^2(\Omega) \times L^2(\Omega) $ be such that:  $ d(x):=\frac{\theta_2}{\theta_1}{\bf 1}_{\Lambda(\theta_1)}\in L^{\infty}(\Omega)$ and that for a.e. $x\in \Omega,$ we have $\theta_1(x)=0 \Rightarrow \theta_2(x)=0.$  We further assume that  assumptions $(\mathcal{P}_1)-(\mathcal{P}_4)$ hold for $\zeta=\theta_1$.
Then for every initial state $(w_1,w_2)\in H$ and for every $\epsilon > 0,$ there are a time $T =T(w_1,w_2,\theta_1,\theta_2,\epsilon)>0$ and  a piecewise static  control
$v(\cdot,t)=v_T(\cdot)\in L^\infty(\Omega)$  such that the respective solution  to (\ref{wave-init}) satisfies:
\begin{equation}\label{1appro}
\|w(T)-\theta_1\|_{H_0^1(\Omega)}+\|w_t(T)-\theta_2\|_{L^2(\Omega)} <\epsilon\cdot
\end{equation}
\end{theorem}
\begin{proof}
From Theorem \ref{thm2-i}, we deduce that for any  $\epsilon>0 $ there is a time $T_1=T_1(w_1,w_2,\theta_1)$ such that the  control: $v(x,t)=v_1(x)=-\frac{\Delta \theta_1+f(\theta_1)}{ \theta_1 } {\bf 1}_{\Lambda(\theta_1)},\;t\in (0,T_1)$ guarantees the following estimate for  the corresponding solution of (\ref{wave-init}):
\begin{equation}\label{estim-thetea-0}
\|w(T_1)-\theta_1\|_{H_0^1(\Omega)}+\|w_t(T_1)\|_{L^2(\Omega)} <\epsilon.
\end{equation}
We will continue to control our initial system (\ref{wave-init}) on $(T_1,T)$ until the achievement of the full target state $\theta=(\theta_1,\theta_2)$, where $T>T_1$ is to be determined. Consider the following system:
\begin{equation}\label{waveS2}
\left\{%
\begin{array}{llll}
  w_{tt}&=&\Delta w + v(x,t) w - h(x) w_t+f(w),
& \hbox{in}\;\;\;\;
\Omega\times (T_1,T)\\
w&=&0, \;  & \hbox{on} \ \partial\Omega\times (T_1,T)\\
w(x,T_1)&=&w(T_1^-),\,  w_{t}(x,T_1)=w_t(T^-_1), & \hbox{in}\,\ \Omega \\
\end{array}%
\right.
\end{equation}
We will use Theorem \ref{thm1c} to reach $(w(T_1^-),\theta_2)$ at a time $T>T_1$ which is close to $T_1.$ For this end, let us observe that by virtue of (\ref{estim-thetea-0}) the system (\ref{waveS2}) can be  approximated by the following one:
\begin{equation}\label{waveS2-theta}
\left\{%
\begin{array}{llll}
  \tilde{w}_{tt}&=&\Delta \tilde{w} + v(x,t)  \tilde{w} - h(x)\tilde{w}_t+f(\tilde{w}),
& \hbox{in}\;\;\;\:
\Omega\times (T_1,T)\\
\tilde{w}&=&0, \;  & \hbox{on} \ \partial\Omega\times (T_1,T)\\
\tilde{w}(x,T_1)&=&\theta_1,\,  \tilde{w}_{t}(x,T_1)=0, & \hbox{in}\,\ \Omega \\
\end{array}%
\right.
\end{equation}
Applying Theorem \ref{thm1c} to system (\ref{waveS2-theta}), we deduce the existence of a static control $v(x,t)=v_2(x)\in L^\infty(\Omega)$  such that the corresponding state $(\tilde{w}(T),\tilde{w}_t(T))$  is close to
$(\theta_1,\theta_2)$ at some $T=T(\theta_1,T_1)$ which is sufficiently close to $T_1^+$.\\
Using the same control $v(x,t)=v_2(x)$ for (\ref{waveS2}), we can see by  Gronwall's inequality and the variation of constants formula that:
$$
\|w(t)-\tilde{w}(t)\| \le \|w(T_1^-)-\theta_1\| e^{C (t-T_1)},\; T_1\le t\le T,
$$
where $ C=L+\|v_2\|_{L^\infty(\Omega)} $ and $L$ is a Lipschitz constant of the function $F: H \rightarrow H;\; y=(w_1,w_2) \mapsto (0,f(w_1)-h w_2)$. Thus
\begin{equation}\label{w-tilde}
\|w(t)-\tilde{w}(t)\|<e^C \epsilon,\; \forall t\in (T_1,T),
\end{equation}
whenever $0<T-T_1<1.$ We deduce that (\ref{1appro}) holds.\\
We conclude that  the initial system (\ref{wave-init}) can be approximately steered  to  $(\theta_1,\theta_2)$
 at $T$ by using the control:
$$
v(x,t)=\left\{
         \begin{array}{ll}
          v_1(x), & \:t\in (0,T_1) \\
          \\
           v_2(x), & t\in (T_1,T)
           \end{array}
       \right.
$$
This  completes the proof.
\end{proof}

\begin{remark}\label{rem2}
According to the proof of Theorem \ref{thm1c}, we can see that  in the case: $\frac{\theta_2}{\theta_1}{\bf 1}_{\Lambda(\theta_1)}\in W^{2,\infty}(\Omega)$, one can  take  the control $v_2(x)=\frac{\theta_2}{(T-T_1)\theta_1}{\bf 1}_{\Lambda(\theta_1)}$  in the time-interval $(T_1,T).$
\end{remark}

\section{Exact controllability }
In this section, we  study the set of target states  that can be exactly achieved at a finite time by the system (\ref{wave-init}) for $n=1$.
The idea in this part consists first, thanks to the continuity of  the Sobolev   embedding $H^1(\Omega) \hookrightarrow
\mathcal{C}^0(\overline{\Omega})$  for $n=1$, in applying  the results of
Section 2 in order to make the state closer to the desired one at a
time $T_1$ with respect to $L^\infty-$norm. Then one can exploit  the results of the exact additive controllability of semilinear wave equation to construct a time $T$ and  a control $v(x,t)$ on $(T_1,T)$ that guarantee the exact
steering of the target state at $T$.\\
In this section, we take $n=1$ and $\Omega=(0,l), \; l>0$.

\subsection{Damped case }
\subsubsection{The case of homogeneous boundary conditions}
In this part, we will study the exact controllability of the one dimensional version of the equation (\ref{wave-init}) evolving in a  time-interval $(0,T)$.\\
For any  $\zeta\in H_0^1(\Omega)\cap H^2(\Omega)$ and $0<t_0<T$, we  consider the following  system:
\begin{equation}\label{wave-semil*}
    \left\{%
\begin{array}{llll}
    \psi_{tt}&=& \Delta \psi+b_\zeta \psi -h(x) \psi_t +f_{\zeta}(\psi) +{\bf 1}_{O} u(x,t), & \hbox{in}\;\;\:\: \Omega\times (t_0,T) \\
    \psi(0,t)&=&\psi(l,t)=0, & \hbox{in}\,   (t_0,T)\\
    \psi(.,t_0)&=&\psi_1,\, \psi_{t}(.,t_0)=\psi_2, & \hbox{in}\, \Omega \\
\end{array}%
\right.
\end{equation}
where $O$ is a sub-domain of $\Omega$,
and let us consider the following property:

$(\mathcal{P}_5) : $   For every $t_0>0$, the system (\ref{wave-semil*}) is exactly null controllable at some time  $T>t_0$  with a control $u(x,t)$  satisfying
  \begin{equation}\label{add-contrl}
   \big ( \int_{t_0}^{T}\|u(\cdot,t)\|_{L^2(O)}^{2}\,dt \big )^{\frac{1}{2}} \leq c_{_{T-t_0}}   \|\psi_1\|_{{ H_0^1(\Omega)}},
\end{equation}
where $c_{_{T-t_0}}>0$ is a  constant depending on $T-t_0$.\\
We refer the reader to \cite{fu,leiv,mari,zha,zua2,wu1,wu2} for  some results on the exact controllability problem for equations like (\ref{wave-semil*}).\\
We are ready to state our first main result of this section.
\begin{theorem}\label{thm3}
Let $n=1, \;(w_1,w_2)\in H_0^{1}(\Omega)  \times L^2(\Omega)-\{(0,0)\}$ and let  $ \theta_1\in  H_0^1(\Omega)\cap H^2(\Omega) $ be such that: $\theta_1 \ne 0, $ a.e in $\overline{O}$ for some open subset $O$ of $\Omega.$ Assume that  assumptions $(\mathcal{P}_1)-(\mathcal{P}_5)$ hold  for  $\zeta=\theta_1$.\\
Then  there exist  $T=T(w_1,w_2,\theta_1)>0$ and a control
$v(\cdot,\cdot)\in L^2(0,T;L^2(\Omega))$   such that the respective solution to the system (\ref{wave-init}) satisfies
 $w(T)=\theta_1$ and $w_t(T)= 0.$

\end{theorem}

\begin{proof}
Let us set  $z=w-\theta_1$ in  the system (\ref{wave-init}). We have:
\begin{equation}\label{Ekwave0}
 z_{tt}=\Delta z + v(x,t)(z+\theta_1) - h(x) z_t+f(z+\theta_1)+\Delta \theta_1,\; \forall t\in (0,T),\; a.e.\;  x\in \Omega.
 \end{equation}
Then for any fixed $0<\epsilon<1,$  it comes from Theorem \ref{thm2-i} that there is a time $T_1>0$ (large enough) such that the  control defined by $v(x,t)=b_{\theta_1}=-\frac{\Delta \theta_1+f(\theta_1)}{ \theta_1} {\bf 1}_{\Lambda(\theta_1)}$ guarantees the following estimate:
\begin{equation}\label{Ekine2}
\|(z(T_1),z_t(T_1))\|_H< \epsilon.
\end{equation}
Let  $T>T_1$, and let us consider the following  system
\begin{equation}\label{Ekwavelin}
    \left\{%
\begin{array}{llll}
    \psi_{tt} &=& \Delta \psi +b_{\theta_1} \psi-h(x) \psi_t+f_{\theta_1}(\psi) +{\bf 1}_{O} u(x,t), & \hbox{in}\;\;\:\: \Omega\times (T_1,T) \\
    \psi(0,t)&=&\psi(l,t)=0, & \hbox{in}\,   (T_1,T)\\
    \psi(.,T_1)&=&z(T_1^-),\, \psi_{t}(.,T_1)=z_t(T_1^-), & \hbox{in}\, \Omega \\
\end{array}%
\right.
\end{equation}
where $u(x,t)$ is an additive control. \\
By assumption,  there exists $T>T_1$ and $u(\cdot)\in L^2(T_1,T;L^2(\Omega))$ satisfying (\ref{add-contrl}) and  such that the respective solution to  system (\ref{Ekwavelin}) satisfies: $(\psi(T),\psi_t(T))=(0,0)$. Then, in order to construct a control that steers   (\ref{wave-init}) to $(\theta_1,0)$, it suffices to look for a control $v(x,t)=v_1(x,t)+b_{\theta_1}(x)$ on $(T_1,T)$ such that:
$$v_1(x,t)  (\psi(x,t)+\theta_1(x))={\bf 1}_{O}  u(x,t).$$
For this purpose, we will show that $\psi(x,t)+\theta_1(x)\neq0, $ a.e.
$x\in O\times (T_1,T),$ and then take for $t\in (T_1,T)$:
$$
v_1(x,t)=\left\{
  \begin{array}{ll}
    \displaystyle\frac{u(x,t)}{ \psi(x,t)+\theta_1(x)}, & \; \mbox{a.e.} \; \;\;\;\;\; x\in O \\
    \\
    0, & \; \mbox{a.e.} \;  x\in \Omega \setminus {O}
  \end{array}
\right.
$$
The solution of (\ref{Ekwavelin}) satisfies the following integral formula:
\begin{equation}\label{Ekcontformula}
\begin{array}{ccc}
(\psi,\psi_{t})(t)&=&S(t-T_1)(z(T_1),z_{t}(T_1))+\int_{T_1}^{t}S(t-\tau) \big ( 0, - h(x)\psi_t(x,\tau)+\\
&&b_{\theta_1} \psi(\tau)+f_{\theta_1}(\psi)(\tau) +{\bf 1}_{O} u(x,\tau) \big ) d\tau,\,\ \forall t\in
[T_1,T].
\end{array}\end{equation}
Since $f_{\theta_1}$ is Lipschitz and vanishes at $0$,  we deduce from the
 formula (\ref{Ekcontformula}) and by using  (\ref{add-contrl}) and (\ref{Ekine2})    that:
$$\|(\psi,\psi_{t})(t)\|_{H}
\leq C  \epsilon+C \int_{T_1}^{T}
\|(\psi,\psi_{t})(\tau)\|_{H} d\tau, \;\forall t\in [T_1,T]\; (\mbox{for some} \;C>0),$$
which by using the Gronwall's inequality gives:
$$\|(\psi,\psi_{t})(t)\|_{H}\leq
C\epsilon, \;\forall t\in [T_1,T], \; (C>0) $$
and so
$$\|\psi(t)\|_{H_0^1(\Omega)}\leq C  \epsilon, \;\forall t\in [T_1,T].$$
This together with  the continuity of the embedding for $n=1$  $H_0^1(\Omega) \hookrightarrow L^\infty(\Omega)$ (see e.g. \cite{adam}) gives:
\begin{equation}\label{Ekequ1}
    \|\psi(t)\|_{L^{\infty}(\Omega)}\leq C  \epsilon,\; (C>0).
\end{equation}
Moreover, since $\theta_1 \ne 0, $ a.e in $\overline{O},$ we deduce from the fact that  the  embedding $H^{1}(O)\hookrightarrow \mathcal{C}^0(\overline{O})$  is continuous  (recall that  $n=1$)  that  $|\theta_1|\geq \mu>0,$ a.e. in $O$. \\
Then, taking $0<\epsilon<\displaystyle\frac{\mu}{2C}$ in (\ref{Ekequ1}),
we deduce  that for all $ t\in (T_1,T) $ we have
\begin{equation}\label{Ekmu}
|\psi(t)+\theta_1| \geq |\theta_1|-|\psi(t)| \ge
\frac{\mu}{2},\;\mbox{a.e. in} \; O\cdot
\end{equation}
Then, one can choose the control $v_1$  as follows:
\begin{equation}\label{Ekcontt}
    v_1(x,t)=\frac{u(x,t)}{\psi(x,t)+\theta_1(x)} {\bf 1}_{O\times (T_1,T)}.
\end{equation}
From (\ref{Ekmu}) and the fact that $u\in L^2(T_1,T;L^2(\Omega))$, it comes that $v_1\in L^2(T_1,T;L^2(\Omega))$.\\
With this  control,   the   system (\ref{Ekwave0})  becomes:
\begin{equation}\label{Ekwaveexact}
    \left\{%
\begin{array}{llll}
    z_{tt} &=& A z+b_{\theta_1}  z - h(x)z_t+f_{\theta_1}(z) +  (z+\theta_1) \frac{u(\cdot,t)}{\psi(\cdot,t)+\theta_1} {\bf 1}_{O}, & \hbox{in}\;\:\:\: \Omega\times (T_1,T) \\
    z(0,t)&=&z(l,t)=0, & \hbox{in}\,   (T_1,T)\\
    z(.,T)&=&z(T_1^-),\, z_{t}(.,T_1)=z_t(T_1^-), & \hbox{in}\, \Omega \\
\end{array}%
\right.
\end{equation}
It is obvious  that $\psi$ is a solution of (\ref{Ekwaveexact}). Let us show that this is the unique one.\\
 Let $z\in H_0^1(\Omega)$ be a solution of (\ref{Ekwaveexact}). The H$\ddot{o}$lder's inequality leads to:
$$
\int_{T_1}^{T}  \|u(s) (z(s)-\psi(s))\|_{L^2(\Omega)} ds \le \|u(\cdot,t)\|_{L^2(T_1,T;L^2(\Omega))} \, \|z(s)-\psi(s)\|_{L^2(T_1,{T};L^\infty(\Omega))} ds$$
$$
\hspace{4.5cm}\le  C \|u(\cdot,t)\|_{L^2(T_1,T;L^2(\Omega))} \, \|z(s)-\psi(s)\|_{L^2(T_1,{T};H_0^1(\Omega))} ds
$$
for some constant $C>0.$\\
This together with (\ref{Ekmu}) and  the variation of constants formula enables us to establish the following inequality:\\

$
\|(z,z_t)(t)-(\psi,\psi_t)(t)\|\le
$
$$
 \big ( C_1 + C_2 \|u(\cdot,t)\|_{L^2(T_1,T;L^2(\Omega))} \big )  \big ( \int_{T_1}^{T} \|(z,z_s)(s)-(\psi,\psi_s)(s)\|^2 ds \big )^{\frac{1}{2}},
$$
for some constants $ C_1, C_2>0.$ As a consequence we have $z(t)-\psi(t)=0$ for all $t\in [T_1,T].$ Then solution of the system (\ref{Ekwave0}) is such that $z(T)=0$ and $z_{t}(T)=0$ and hence $w(T)=\theta_1$ and $w_{t}(T)=0$. \\
We conclude that the control defined by:
$$
v(\cdot,t)=\left\{
         \begin{array}{ll}
           -\frac{\Delta \theta_1+f(\theta_1)}{ \theta_1}{\bf 1}_{\Lambda(\theta_1)}, & \:t\in (0,T_1) \\
  \frac{u(\cdot,t)}{\psi(\cdot,t)+\theta_1} {\bf 1}_{O}-\frac{\Delta \theta_1+f(\theta_1)}{ \theta_1 } {\bf 1}_{\Lambda(\theta_1)}, & t\in (T_1,T)
\end{array}
       \right.
$$
steers the system (\ref{wave-init}) from the initial state
$(w_{_1},w_{_2}) $ to the desired one $(\theta_1,0)$  at $T.$
\end{proof}

\subsubsection{The case of nonhomogeneous boundary conditions}
Here, we intend to study the possibility of achieving a full state $\theta=(\theta_1,\theta_2)$ for the following one dimensional system with nonhomogeneous Dirichlet boundary conditions:
\begin{equation}\label{NHwave-init}
\left\{%
\begin{array}{llll}
  w_{tt}&=&\Delta w + v(x,t)  w -h(x) w_t + f(w),
& \hbox{in}\:\;\;\;
 \Omega \times (0,T)\\
w(0,t)&=&\sigma_1, \ w(l,t)=\sigma_2,   & \hbox{in} \  (0,T)\\
w(x,0)&=&w_{_1},\,  w_{t}(x,0)=w_{_2}, & \hbox{in}\,\ \Omega
\end{array}%
\right.
\end{equation}
with the same assumptions as in (\ref{wave-init}), and $\sigma_1, \sigma_2\in \mathbf{R}.$

For any  $\zeta\in  H^2(\Omega)$ satisfying the  compatibility condition $\zeta|_{\partial \Omega}=\sigma:=(\sigma_1,\sigma_2)$, we  consider the following  system  with additive globally distributed control:
\begin{equation}\label{NHwave-semil*}
    \left\{%
\begin{array}{llll}
    \psi_{tt}&=& \Delta \psi+b_\zeta \psi -h(x) \psi_t +f_{\zeta}(\psi) + u(x,t), & \hbox{in}\;\;\:\: \Omega\times (t_0,T) \\
    \psi(0,t)&=& \psi(l,t)=0, & \hbox{in}\, (t_0,T)\\
    \psi(.,t_0)&=&\psi_1,\, \psi_{t}(.,t_0)=\psi_2, & \hbox{in}\, \Omega \\
\end{array}%
\right.
\end{equation}
where $0<t_0<T$.\\
In the sequel,  we will  consider the case of exact steering  of (\ref{NHwave-semil*})  from an initial state $\psi_0=(\psi_1,\psi_2)$ to a target state $\psi_d=(\psi_1^d,\psi_2^d)$   under a control $u\in L^2(t_0,T;L^2(\Omega))$ that satisfies the following  bound inequality with respect to initial and target states:
\begin{equation}\label{NHadd-contrl}
   \big ( \int_{t_0}^{T}\|u(\cdot,t)\|_{L^2(\Omega)}^{2}\,dt \big )^{\frac{1}{2}} \leq c_{_{T-t_0}}  C\big ( \|\psi_0\|_{{ H}},\|\psi_d\|_{{ H}} \big),
\end{equation}
where $c_{_{T-t_0}}>0$ is a bounded function of $T-t_0$  and  $ C\big ( \|\psi_0\|_{{ H}},\|\psi_d\|_{{ H}}\big )>0$ is a function of $\|\psi_0\|_{{ H}}$ and $\|\psi_d\|_{{ H}}$.
In the next theorem, we will consider the case where $t_0$ and  $T$ are close to each other, which  may be linked  to the question of exact controllability in short time (see \cite{curt,kom,mart}).
Note that, since the control acts in all of $\Omega$, the exact controllability of (\ref{NHwave-semil*}) holds in any time $T > 0$.
This may be deduced from  the case of the linear version of (\ref{NHwave-semil*}) (i.e. $f_{\zeta}=0$)  and the fact that, in the case of globally distributed control, the nonlinearity can be suppressed in a trivial way.

We will again proceed as in the case of homogeneous boundary conditions, but here we need to use an auxiliary   $\zeta\in H^2(\Omega)$ such that $w-\zeta$ is the solution of a system like (\ref{NHwave-semil*}) with the condition that $ \zeta\ne 0$ a.e. on $\overline{\Omega}.$ This is why we deal with nonhomogeneous boundary conditions. Moreover, unlike the case of homogeneous BC, here the estimate (\ref{NHadd-contrl}) involves the term $ \|\psi_d\|_{{ H}}$, so we require more than the null exact controllability of the auxiliary system (\ref{NHwave-semil*}). \\
 For any $\zeta\in H_\sigma^2(\Omega):= \{\zeta \in H^2(\Omega):\; \zeta|_{\partial\Omega}= \sigma\}$, we consider the following assumption:

$(\mathcal{P}_6):  \;$ The system (\ref{NHwave-semil*}) is exactly controllable   at any $T>t_0>0$  large enough,   with a control   satisfying (\ref{NHadd-contrl}).\\
Let us also introduce the affine space $H_\sigma^{1}(\Omega):=\{\zeta \in H^1(\Omega):\; \zeta|_{\partial\Omega}= \sigma\}$.

We have:
\begin{theorem}\label{NHthm3}
Let $n=1, \;(w_1,w_2)\in H_\sigma^{1}(\Omega)  \times L^2(\Omega)$ and let  $ (\theta_1,\theta_2)\in H_\sigma^{1}(\Omega)  \times L^2(\Omega). $ If assumptions $(\mathcal{P}_1)-(\mathcal{P}_4)$ and $(\mathcal{P}_6)$ hold for some $\zeta\in H_\sigma^2(\Omega)$ such that $ \zeta \ne 0, $ a.e in $\overline{\Omega}=[0,l], $   then  there exist a time $T=T(w_1,w_2,\zeta)>0$ and a control
$v(\cdot,\cdot)\in L^2(0,T;L^2(\Omega))$   such that the corresponding solution of the system (\ref{NHwave-init}) satisfies
 $(w(T),w_t(T))= (\theta_1,\theta_2).$
\end{theorem}
\begin{proof}
Consider the following system:
\begin{equation}\label{NHEkwave0}
    \left\{%
\begin{array}{llll}
 z_{tt}&=&\Delta z + v(x,t)(z+\zeta) - h(x) z_t+f(z+\zeta)+\Delta \zeta, & \hbox{in}\;\;\:\: \Omega\times (0,T)\\
   z(0,t)&=&z(l,t)= 0, & \, \mbox{in} \; (0,T)\\
    z(\cdot,0)&=&(w_1-\zeta,w_2), & \, \mbox{in} \; \Omega
\end{array}%
\right.
\end{equation}
For any fixed $0<\epsilon<1,$  it comes from the proof of Theorem \ref{thm2-i} that there is a time $T_1>0$ (large enough) such that the  control defined by $v(x,t)=b_\zeta=-\frac{\Delta \zeta+f(\zeta)}{ \zeta} {\bf 1}_{\Lambda(\zeta)}$ guarantees the following estimate:
\begin{equation}\label{NHEkine2}
\|(z(T_1),z_t(T_1))\|_H< \epsilon.
\end{equation}
Let  $T>T_1$,
 and let us consider the following additive-control system:
\begin{equation}\label{NHEkwavelin}
    \left\{%
\begin{array}{llll}
    \psi_{tt} &=& \Delta \psi +b_\zeta \psi-h(x) \psi_t+f_{\zeta}(\psi) + u(x,t), & \hbox{in}\;\;\:\: \Omega\times (T_1,T) \\
    \psi(0,t)&=& \psi(l,t)=0, & \hbox{in}\,  (T_1,T)\\
    \psi(.,T_1)&=&z(T_1^-),\, \psi_{t}(.,T_1)=z_t(T_1^-), & \hbox{in}\, \Omega \\
\end{array}%
\right.
\end{equation}
where $u(x,t)$ is the additive control.\\
By assumption,  there exists $u(\cdot)\in L^2(T_1,T;L^2(\Omega))$ satisfying (\ref{NHadd-contrl}) and is   such that the respective solution to  system (\ref{NHEkwavelin}) satisfies: $(\psi(T),\psi_t(T))=(\theta_1-\zeta,\theta_2)$. \\
Then, in order to construct a control that steers   (\ref{NHwave-init}) to $(\theta_1,\theta_2)$, it suffices to build a control $v(x,t)=v_1(x,t)+b_\zeta(x)$ on $(T_1,T)$ such that:
$$v_1(x,t)  (\psi(x,t)+\zeta(x))=  u(x,t),$$
which may be done as in the proof of Theorem \ref{thm3} by observing (thanks to  estimate (\ref{NHadd-contrl})) that the additive steering control $u(x,t)$ satisfies:
$$\big ( \int_{T_1}^{T}\|u(\cdot,t)\|_{L^2(O)}^{2}\,dt \big )^{\frac{1}{2}} < \epsilon,$$
whenever $T$ is sufficiently close to $T_1$. \\
Hence the  respective solution to (\ref{NHEkwavelin}) satisfies the estimate:
\begin{equation}\label{NHEkequ1}
    \|\psi(t)\|_{L^{\infty}(\Omega)}\leq C  \epsilon,\; (C>0).
\end{equation}
This enables us to show that $\psi(t)+\zeta\ne 0$, a.e. in $ \Omega,$ so that, one can consider the control $v_1$  defined by:
\begin{equation}\label{NHEkcontt}
    v_1(x,t)=\frac{u(x,t)}{\psi(x,t)+\zeta(x)},
\end{equation}
which renders the   system (\ref{NHEkwave0})  equivalent to the following one:
\begin{equation}\label{NHEkwaveexact}
    \left\{%
\begin{array}{llll}
    z_{tt} &=& A z+b_\zeta  z - h(x)z_t+f_{\zeta}(z) +  (z+\zeta) \frac{u(\cdot,t)}{\psi(\cdot,t)+\zeta}, & \hbox{in}\;\:\:\: \Omega\times (T_1,T) \\
    z(0,t)&=&z(l,t)=0, & \hbox{in}\,  (T_1,T)\\
    z(.,T)&=&z(T_1^-),\, z_{t}(.,T_1)=z_t(T_1^-), & \hbox{in}\, \Omega \\
\end{array}%
\right.
\end{equation}
whose unique solution is the same as the one of (\ref{NHEkwavelin}). In other words, $z=\psi$. Then the solution of the system (\ref{NHEkwave0}) is  such that $z(T)=\theta_1-\zeta$ and $z_{t}(T)=\theta_2. $ \\
Let us now set: $w=z+\zeta.$ Then $w$ is the unique solution of the system (\ref{NHwave-init}) and we have $(w(T),w_{t}(T))=(\theta_1,\theta_2)$.
We conclude that the control defined by:
$$
v(\cdot,t)=\left\{
         \begin{array}{ll}
           -\frac{\Delta \zeta+f(\zeta)}{ \zeta}{\bf 1}_{\Lambda(\zeta)}, & \:t\in (0,T_1) \\
  \frac{u(\cdot,t)}{\psi(\cdot,t)+\zeta} -\frac{\Delta \zeta+f(\zeta)}{ \zeta } {\bf 1}_{\Lambda(\zeta)}, & t\in (T_1,T)
\end{array}
       \right.
$$
guarantees the exact steering of the system (\ref{NHwave-init}) from the initial state
$(w_{_1},w_{_2}) $ to  $(\theta_1,\theta_2)$  at $T.$
\end{proof}

\begin{remark}\label{NHrem*}
 If the assumptions of Theorem \ref{NHthm3} hold for $\zeta=\theta_1,$ then the exact controllability required for the system   (\ref{NHwave-semil*}) can be restricted to target states of the form: $\psi_d=(0,\psi_2^d)$.

\end{remark}

\subsection{Undamped case}
In this subsection, we will establish an exact controllability result for an uniform time $T$ when dealing with undamped equation.
 We consider the  following  one dimensional undamped equation:
\begin{equation}\label{wwave-init}
\left\{%
\begin{array}{llll}
  w_{tt}&=&\Delta w + f(w)+ v(x,t)  w,
& \hbox{in} \;\;\;\;\Omega \times (0,T)\\
w(0,t)&=&w(l,t)=0, \;  & \hbox{in} \  (0,T)\\
w(x,0)&=&w_{_1},\,  w_{t}(x,0)=w_{_2}, & \hbox{in}\,\ \Omega \\
\end{array}%
\right.
\end{equation}
where $f$ is globally Lipschitz. In the context of additive controls, Zuazua \cite{zua2} has considered the exact internal controllability of the  one dimensional version of following semilinear system: \begin{equation}\label{SL}
y_{tt}=\Delta y+f(y)+{\bf 1}_O u(x,t),\;\; (O\subset\Omega)
\end{equation}
The  multidimensional case  has been treated in \cite{zha}.\\
For any  $\zeta\in H_0^1(\Omega)\cap H^2(\Omega)$ and $0<t_0<T$, we  consider the following  system:
\begin{equation}\label{wave-semil*und}
    \left\{%
\begin{array}{llll}
    \psi_{tt}&=& \Delta \psi+b_\zeta \psi  +f_{\zeta}(\psi) +{\bf 1}_{O} u(x,t), & \hbox{in}\;\;\:\: \Omega\times (t_0,T) \\
    \psi(0,t)&=&\psi(l,t)=0, & \hbox{on}\,  \partial\Omega\times (t_0,T)\\
    \psi(.,t_0)&=&\psi_1,\, \psi_{t}(.,t_0)=\psi_2, & \hbox{in}\, \Omega \\
\end{array}%
\right.
\end{equation}
 which we will assume to be  exactly null controllable  at  $T>t_0$ for $t_0$ small enough and $T$ large enough (i.e. for $t_0<\alpha< T$ for some $\alpha>0$) with controls $u\in L^2(t_0,T;L^2(O))$ such that:
\begin{equation}\label{add-contrl-und}
   \big ( \int_{t_0}^{T}\|u(\cdot,t)\|_{L^2(O)}^{2}\,dt \big )^{\frac{1}{2}} \leq C    \|(\psi_1,\psi_2)\|_{{ H}},
\end{equation}
for some positive constant $C=C_{T}$  depending only on  $T, $ for  $t_0>0$ small enough.

The next theorem states  our second main result of this section.

\begin{theorem}\label{thm-exact2}
Let $n=1, \; (w_1,w_2)\in H_0^{1}(\Omega)  \times L^2(\Omega)-\{(0,0)\}, $ and  let $\theta_1\in H_0^{1}(\Omega) \cap H^2(\Omega)$ be such that  $\theta_1(x) \ne 0, $ a.e. $x\in\overline{O}$ for some open subset $O$ of $\Omega$ and $b_{\theta_1}\in L^\infty(\Omega)$.

If for some $\alpha>0,$ the target state $\theta:=(\theta_1,0)$ is approximately reachable  at a small  time $0<T_1<\alpha$ with a control $v_{T_1}  \in L^2(0,T_1;L^2(\Omega))$, and if the system (\ref{wave-semil*und}) is exact  null controllable at $T>\alpha$ for $\zeta=\theta_1$  with a control $u$ satisfying  (\ref{add-contrl-und}), then  there exists  a control
$v  \in L^2(0,T;L^2(\Omega))$  such that the corresponding solution of (\ref{wwave-init}) satisfies:
 $w(T)=\theta_1$ and $w_t(T)=0.$
\end{theorem}

\begin{proof}
Let $T>\alpha$, and let us set  $z=w-\theta_1$ in  the system (\ref{wwave-init}). Then we have
\begin{equation}\label{wEkwave0}
 z_{tt}=\Delta z + v(x,t)(z+\theta_1) +f(z+\theta_1)+\Delta \theta_1,\; x\in \Omega,\; t\in (0,T).
 \end{equation}
For any fixed $0<\epsilon<1,$ there are  $T_1\in (0,\alpha)$ small enough and  a control $v_{T_1}$ that provide  the following estimate: \begin{equation}\label{wEkine2}
\|(z(T_1),z_t(T_1))\|_H< \epsilon.
\end{equation}
Letting $v(x,t)=v_1(x,t)+b_{\theta_1}$, we deduce that $z$ satisfies  the following homogeneous  equation:
\begin{equation}\label{wEkwave0*}
 z_{tt}=\Delta z + b_{\theta_1} z + v_1(x,t)(z+\theta_1) +f_{\theta_1}(z),\; x\in \Omega,\; t\in (0,T).
 \end{equation}
Let  us now consider the following  system:
\begin{equation}\label{wEkwavelin}
    \left\{%
\begin{array}{llll}
    \psi_{tt}&=& \Delta \psi+b_{\theta_1}\psi +f_{\theta_1}(\psi) +{\bf 1}_{O} u(x,t), & \hbox{in} \;\;\:\: \Omega\times (T_1,T) \\
    \psi(0,t)=\psi(l,t)&=&0, & \hbox{on}\,  \partial\Omega\times (T_1,T)\\
    \psi(.,T_1)&=&z(T_1^-),\, \psi_{t}(.,T_1)=z_t(T_1^-), & \hbox{in}\, \Omega \\
\end{array}%
\right.
\end{equation}
where $u(x, t)$ is an additive control. \\
By assumption, there exists a control $u(\cdot,\cdot)\in L^2(T_1,T;L^2(\Omega))$ such that $\psi(T)=\psi_t(T)=0$ and
\begin{equation}\label{wEkestimcontrl}
    \int_{T_1}^{T}\|u(\cdot,t)\|_{L^2(O)}^{2}\,dt\leq C\|(z(T_1),z_t(T_1))\|^2_{{ H}},
\end{equation}
where the positive constant $C$ can be chosen independent of $T_1$.
 Then, in order to construct a control that steers   (\ref{wwave-init}) to $(\theta_1,0)$ at $T$, it suffices to look for a control $v(x,t)$ on $(T_1,T)$ such that:
$$v(x,t)  (\psi(x,t)+\theta_1(x))={\bf 1}_{O}  u(x,t),\; \; a.e. \; \; x\in \Omega.$$
For the remainder part, it suffices to reproduce the corresponding part in the proof of Theorem \ref{thm3} to deduce that the state $(\theta_1,0)$ can be  exactly achieved using the following control:
$$
v(\cdot,t)=\left\{
         \begin{array}{ll}
           v_{_{T_1}}(\cdot,t), & \:t\in (0,T_1) \\
           \\
         b_{\theta_1}+ \frac{u(\cdot,t)}{\psi(\cdot,t)+\theta_1} {\bf 1}_{O}, & t\in (T_1,T)
\end{array}
       \right.
$$

\end{proof}

\begin{remark}
 The results of Theorems \ref{thm3} $\&$ \ref{thm-exact2} can be extended to several dimension in hight energy spaces (see  \cite{ouz14} for the bilinear case).

 \end{remark}

\subsection{Example}
Here, we will present an illustrating example.
Let us consider the following semilinear and linear systems respectively with additive control:
\begin{equation}\label{wave-semil**}
    \left\{%
\begin{array}{llll}
    \psi_{tt}&=& \Delta \psi +\mu(x)\psi -h(x) \psi_t+
    k(\psi) + {\bf 1}_{O} u(x,t), & \hbox{in}\, Q_T:=\Omega\times (0,T) \\
    \psi(0,t)&=&\psi(l,t)=0, & \hbox{in}\,  (0,T)\\
    \psi(.,0)&=&\psi_1,\, \psi_{t}(.,0)=\psi_2, & \hbox{in}\, \Omega \\
\end{array}%
\right.
\end{equation}
and
\begin{equation}\label{wave-lin1}
    \left\{%
\begin{array}{llll}
   \varphi_{tt} &=& \Delta \varphi+\mu(x)\varphi  +{\bf 1}_{O} u_0(x,t), & \hbox{in}\;Q_T \\
    \varphi(0,t)&=&\varphi(l,t)=0, & \hbox{in}\,  (0,T)\\
    \varphi(.,0)&=&\varphi_1,\, \varphi_{t}(.,0)=\varphi_2, & \hbox{in}\,\; \Omega \\
\end{array}%
\right.
\end{equation}
where  $O$ is an open subset of $\Omega,$ the functions $\mu $ and $h$ are such that $ \mu, h\in L^\infty(\Omega), $ the nonlinear term  $k: R \to R$  is Lipschitz,  $u_0(x,t)$ and $u(x,t)$ are the additive controls and belong to $L^2(O\times (0,T))$.

Let us first examine  the exact controllability of (\ref{wave-semil**}). For this end, we start with proving that under the assumption  of  exact controllability of the linear part (\ref{wave-lin1}), the  semilinear system (\ref{wave-semil**}) is exactly controllable  over the same time interval as the linear version (\ref{wave-lin1}).

The following elementary controllability result for the system (\ref{wave-semil**}) is
sufficient for our purpose.
\begin{lem}\label{lem-add}
Assume that:\\
(i) supp$(h)\subset O,$ \\
and \\
(ii)  for all $y\in L^2(\Omega),\; $ we have supp$(k\circ y)\subset O$.\\
If the linear system (\ref{wave-lin1}) is null exactly controllable with a control satisfying (\ref{add-contrl}), then so is the semilinear system (\ref{wave-semil**}).
\end{lem}
\begin{proof}
Let us consider the system (\ref{wave-semil**}) and the following one:
\begin{equation}\label{wave-lin}
    \left\{%
\begin{array}{llll}
   \varphi_{tt} &=& \Delta \varphi+\mu(x)\varphi  +{\bf 1}_{O} u_0(x,t), & \hbox{in}\; Q_T \\
    \varphi(0,t)&=&\varphi(l,t)=0, & \hbox{in}\, (0,T)\\
    \varphi(.,0)&=&\psi_1,\, \varphi_{t}(.,0)=\psi_2, & \hbox{in}\,\; \Omega \\
\end{array}%
\right.
\end{equation}
For any couple of control $u_0(x,t)$ and corresponding solution $\varphi$ of (\ref{wave-lin}), we consider the control:
$u(x,t)=u_0(x,t)+h(x)\varphi_t(x,t)-k(\varphi(x,t))$.\\
From the assumptions on $u_0(x,t), \;h$ and $k$ we can see that $u(x,t)$  satisfies (\ref{add-contrl}) and ${\bf 1}_O u(x,t)={\bf 1}_O u_0(x,t)+h(x)\varphi_t(x,t)-k(\varphi(x,t)).$ \\
Then with the control  $u(x,t)$, the system (\ref{wave-semil**}) takes the form:
$$
    \left\{%
\begin{array}{llll}
    \psi_{tt}&=& \Delta \psi+\mu(x)\psi -h(x) (\psi_t-\varphi_t) + \big ( k(\psi)-k(\varphi) \big ) +{\bf 1}_{O} u_0(x,t), & \hbox{in}\; Q_T\\
    \psi(0,t)&=&\psi(l,t)=0, & \hbox{in}\,  (0,T)\\
    \psi(.,0)&=&\psi_1,\, \psi_{t}(.,0)=\psi_2, & \hbox{in}\,\; \Omega \\
\end{array}%
\right.
$$
which admits $\varphi$ as a particular solution, and  by uniqueness it comes $\psi=\varphi$. Hence the null exact controllability of the semilinear system (\ref{wave-semil**}) follows from the one of its linear part (\ref{wave-lin}).
\end{proof}
Let us now describe our illustrative example. Consider the system (\ref{wave-init}) with  $h={\bf 1}_O $ for some proper open subset $ O=(a,b)$ of $ \Omega=(0,3)$ such that $[1,2]\subset O$ and let $f$ be such that $f(y)(x)=k(y(x)), $ where $k(s)=-s    \alpha(s)\;$ and $\alpha(s)= - c\: (s-a)^2(s-b)^2{\bf 1}_{O},\:  c>0.$ Here, the function $k : \mathbf{R}\to \mathbf{R}$ is $\mathcal{C}^1$ and supp$(k)\subset O$, so $k$ is  Lipschitz.\\
Now in order to define our target state, consider the function defined by: $$\xi_1(x)=\left\{
  \begin{array}{ll}
  x  , & \; x\in [0,1] \\
      \big ( 1-(x-1)(x-2) \big  )e^{(x-1)^2(x-2)^2}, & \; x\in [1,2] \\
   (  -x+3), & \; x\in [2,3]
  \end{array}
\right.
$$
and for each $\eta>0$ we set $\xi_\eta=\eta\xi_1$, and consider the target state $(\theta_1,\theta_2)=(\xi_\eta,-\Delta \xi_\eta).$ We have $(\theta_1,\theta_2)\in \big ( H_0^1(0,3)\cap H^2(0,3) \big ) \times  L^2(0,3)$ and $ \frac{\Delta \theta_1}{\theta_1}\in L^\infty(0,3)$.
Moreover for $ \eta>0$ small enough i.e. $\eta< \frac{a}{\|\xi_1\|_{L^\infty(\Omega)}}, $ we have $k(\xi_\eta(x))=0,\; \forall x\in\Omega $ so that $b_{\theta_1}(x)=-\frac{\Delta \theta_1}{\theta_1}=\frac{ \theta_2}{\theta_1}.$\\
Let us  show that (\ref{coe}) holds. For this end, let us write: $
b_{\theta_1}(x)y(x)+f(y)(x)=p(x)y(x)+ q(y(x)) $ with $p(x)=b_{\theta_1}(x)-\|b_{\theta_1}\|_{L^\infty(\Omega)} {\bf 1}_{(1,2)} \le 0$ and $q(s)=k(s)+ s \|b_{\theta_1}\|_{L^\infty(\Omega)} {\bf 1}_{(1,2)}.$ \\
Observing that $sq(s)=sk(s)\le0$ for all $ s\in O\setminus(1,2)$, we can see that  $ s q(s)  \le 0,\; \forall s\in \mathbf{R}, $ whenever  $c> \|b_{\xi_1}\|_{L^\infty(\Omega)} \; \displaystyle\max_{s\in (1,2)}  \frac{1}{(s-a)^2(s-b)^2}.$\\
From  \cite{fu,teb}, we deduce that the estimate (\ref{coe}) holds. Moreover,   we have $\langle b_{\theta_1}(x) y+f(y),y\rangle\le 0,\; y\in L^2(\Omega).$ Then according to Theorem \ref{thm2}, one can approximately achieve $(\theta_1,\theta_2)$ (for $\eta< \frac{a}{\|\xi_1\|_{L^\infty(\Omega)}}$)
  using  the control:
$$
v(x,t)=\left\{
         \begin{array}{ll}
          v_1(x)= -\frac{\Delta \xi_1}{ \xi_1 }{\bf 1}_{O}, & \:t\in (0,T_1) \\
          v_2(x), & t\in (T_1,T_1')
\end{array}
       \right.
$$
for $T_1$ large enough ($T_1\to+\infty$),  $T_1'$ sufficiently close to $T_1$ and for $a, b$ and $c$ satisfying the above mentioned conditions. \\
Here, the control $v_2\in W^{2,\infty}(\Omega)$ is  an approximation of $-\frac{\Delta \xi_1}{(T_1'-T_1)\xi_1}{\bf 1}_{O}$ in $L^2(\Omega)$.

Let us  establish the null exact controllability of the  additive-control system (\ref{Ekwavelin}).
From the definition of $\theta_1, h$ and $f$, it is clear that the assumptions of Lemma \ref{lem-add} are satisfied for $\mu=b_{\theta_1}.$   Moreover, we know that (see \cite{mart,zha,zua2}) there exist a $T>T'_1$ and a control $u_0\in L^2(T'_1,T;L^2(O))$ satisfying (\ref{add-contrl}) for $t_0=T_1'$ that steers the linear system:
\begin{equation}\label{lin-ex}
 \left\{%
\begin{array}{llll}
    \varphi_{tt} &=& \Delta \varphi +b_{\theta_1} \varphi+{\bf 1}_{O} u_0(x,t), & \hbox{in}\;\;\:\: \Omega\times (T'_1,T) \\
    \varphi(0,t)&=&\varphi(3,t)=0, & \hbox{in}\,   (T'_1,T)\\
    \varphi(.,T_1')&=&w({T'_1}^-)-\theta_1,\, \varphi_{t}(.,T'_1)=w_t({T'_1}^-), & \hbox{in}\, \Omega \\
\end{array}%
\right.
\end{equation}
to $(0,0)$ at $T$. Then it follows from Lemma \ref{lem-add} that the control $u(x,t)=u_0(x,t)+h(x)\varphi_t-f_{\theta_1}(\varphi)$ guarantees the null exact steering of system (\ref{Ekwavelin}) to $(0,0)$ and satisfies (\ref{add-contrl}). Then applying  Theorem  \ref{thm3}, we deduce that the control:
$$
v(x,t)=\left\{
         \begin{array}{ll}
           -\frac{\Delta \xi_1}{ \xi_1 }{\bf 1}_{O}, & \:t\in (0,T_1) \\
           v_2(x), & \:t\in (T_1,T'_1) \\
          \big ( \frac{u(x,t)}{\psi(x,t)+\theta_1(x)} -\frac{\Delta \xi_1}{ \xi_1 } \big ) {\bf 1}_{O}, & t\in (T_1',T)
\end{array}
       \right.
$$
guarantees the exact steering of system (\ref{wave-init}) to  $(\theta_1,0)$, where $\psi$ is the solution of system (\ref{Ekwavelin}) corresponding to the control $u(x,t)$.



\end{document}